\documentclass{amsart}

\usepackage{amsfonts,amsthm,amssymb}
\usepackage{amsmath, amsthm, amssymb,amsfonts}
\usepackage{mathrsfs}
\usepackage{hyperref}
\usepackage{graphicx}
\usepackage{color}
\makeatletter
\def\LaTeX{\leavevmode L\raise.42ex
   \hbox{\kern-.3em\size{\sf@size}{0pt}\selectfont A}\kern-.15em\TeX}
\makeatother

\newcommand{\BibTeX}{{\rm B\kern-.05em{\sc
i\kern-.025emb}\kern-.08em\TeX}}
\newtheorem{col}{Corollary}[section]
\newtheorem{thm}{Theorem}[section]
\newtheorem{lem}[thm]{Lemma}
\newtheorem{defn}[thm]{Definition}

\newtheorem{remark}[thm]{Remark}
\newtheorem*{Vario}{Variational Problem}

\def\N{\mathbb N}

\def\S{\mathbb S}

\def\d{\mathrm d}
\def\G{\mathcal G}
\def\H{\mathcal H}
\def\M{\mathcal M}
\def\h{\mathscr H}
\def\P{\mathbb P}

\def\Y{\mathcal Y}
\def\RD{\mathscr R}

\newcommand{\dx}{\;\mathrm{d}}
\DeclareMathOperator{\tr}{trace}

\begin{document}

\title[Splines for Radon transform on compact Lie groups]{
Generalized splines for  Radon transform on compact Lie groups with applications to crystallography}

\maketitle

\begin{center}

\author{Swanhild Bernstein}
\footnote{TU Bergakademie Freiberg, Institute of Applied Analysis, Germany; swanhild.bernstein@math.tu-freiberg.de}

\author{Svend Ebert}
\footnote{TU Bergakademie Freiberg, Institute of Applied Analysis; Germany; svend.ebert@math.tu-freiberg.de}

\author{Isaac Z. Pesenson }\footnote{ Department of Mathematics, Temple University,
 Philadelphia,
PA 19122, USA; pesenson@math.temple.edu. The author was supported in
part by the National Geospatial-Intelligence Agency University
Research Initiative (NURI), grant HM1582-08-1-0019. }

\end{center}
\begin{abstract}

The Radon transform $\mathcal{R}f$ of functions $f$ on $SO(3)$ has recently been applied extensively in texture analysis, i.e. the analysis of preferred crystallographic orientation. In practice one has to determine the  orientation probability density function  $f\in L_{2}(SO(3))$   from $\mathcal{R}f\in L_{2}(S^{2}\times S^{2})$ which is known only on a discrete set of points.  Since one has only partial information about $\mathcal{R}f$ the inversion of the Radon transform becomes an ill-posed inverse problem.
Motivated by this problem  we define a new notion of the Radon transform  $\mathcal{R}f$ of functions $f $ on  general  compact Lie groups and introduce two  approximate inversion algorithms which utilize our previously developed generalized variational splines on manifolds.
 Our new algorithms fit very well
to the application of Radon transform on $SO(3)$ to texture analysis.
 \end{abstract}



\section{Introduction}

The objective  of the present paper is to introduce Radon transform on compact Lie groups and to show how spline interpolation can be used for approximate inversion of such  transform on general compact Lie groups and in particular on the group of rotations $SO(3)$.

The Radon $ \RD $ transform on a compact $\G$ Lie group associated with a closed subgroup $\h$   assigns  to a smooth function $f$ integrals over submanifolds  of  $\G$ which have the form $x\h y^{-1}$ where $x, y \in \G$.   An important fact is that  the Radon transform $\RD f$ is  a function on $\G /\h\times \G /\h$ and not merely on $\G\times\G$ as one could expect (see below for details).  The typical problem is to reconstruct $f$ knowing $\RD f$. In practice one has  only integrals over manifolds $\{\M_{\nu}\}_{1}^{N}$  from a {\bf  finite} subfamily. In this situation inversion becomes an ill posed problem. Our objective is to consider appropriate regularization of this ill posed problem and to develop a  method  for  approximate inversion of the Radon transform.

In our approach  to approximate inversion of the Radon transform (section 5)  we consider inversion as an interpolation problem. Namely,  given a set of integrals of an $f$ over a finite family $\M=\{\M_{\nu}\}_{1}^{N}$ of submanifolds, we find a "smoothest" function  which has the same set of integrals as $f$ over submanifolds from the family $\M$.

Our work is inspired by recent applications of
the Radon transform  on $SO(3)$  to  texture analysis, i.e. the analysis of preferred crystallographic orientation. 

The orientation probability density function (ODF) representing the probability law of random orientations of crystal grains by volume is a major issue. In $x$-ray or neutron diffraction experiments spherical intensity distributions are measured which can be interpreted in terms of spherical probability distributions of distinguished crystallographic axes. A first mathematical description of the problem was given in \cite{bunge}.

 One equips a crystal with an inner orthogonal coordinate system $\{e_1,e_2,e_3\}$. Additionally one distinguish an outer orthogonal coordinate system $\{u_1,u_2,u_3\}$ related to the specimen. The orientation of a crystal in the specimen is defined by the unique rotation $\gamma\in SO(3)$ which maps the inner coordinate system to the outer one, i.e. $ge_i=u_i$ for $i=1,2,3$.  Note that in this model  we neglect the spherical symmetries of the crystal. 

The function of interest  is the \emph{orientation density function (ODF) $f\in L^2(SO(3))$} that is the probability measure on $SO(3)$. Hence the function value $f(g)$ gives the portion of crystals in the specimen with orientation $g$.

The practical measurement sends a beam through the specimen coming from the direction $h\in S^2$ and measures the intensity, emitted from the specimen in the direction $r\in S^2$. One can interpret the result as the integral over all orientations $g\in SO(3)$ with $g\cdot h=r;\>\>\>h,r\in S^{2}$. The set $C_{h,r}=\{g\in SO(3) :  \,g\cdot h=r;\>\>\>h,r\in S^{2}\}$ of those orientations is called  a great circle  in $SO(3)$. 

Since $SO(2)$ is the stabilizer of $\xi_0\in S^2$, where $\xi_0$ is the north pole one has
\begin{align}
    C_{h,r} &= h' \,SO(2) \, r'^{-1}:=\{h' g r'^{-1},\,g\in SO(2)\}\qquad h',r'\in SO(3),
\end{align}
where $h',r'\in SO(3)$ satisfy $h'\cdot\xi_0=h$ and $r'\cdot\xi_0=r$. Here $(x,y)$ is a fixed point in $S^{2}\times S^{2}$.

\begin{defn}\label{def:RD_on_SO(3)}
    The Radon transform of a continuous complex-valued function $f$ on $SO(3)$ is a function on $S^{2}\times S^{2}$ which is defined by the formula 
    \begin{align}
        \RD f(x,y) &= \int_{C_{x,y}} f(g) \dx x .
    \end{align}

This transform $\RD$ can be extended to all functions in $ L^2(SO(3))$.

\end{defn}

From the very definition $\RD f$ is a function on $S^{2}\times S^{2}$. Moreover, it is  shown that $\RD f$ is in the kernel of the  Darboux-type differential operator (also called ultra-hyperbolic operator). In other words,
 one has 
 $$
 \Delta_{x} \RD f (x,y)= \Delta_{y} \RD f (x,y),\>\>\>(x,y)\in S^{2}\times S^{2}.
 $$  
 It is interesting to note that this condition is similar to the classical condition of F. John \cite{J} for $X$-ray transform in $\mathbb{R}^{3}$.

The  Darboux-type  equation shows  that $\RD f$ belongs to the span of  the tensorial product of spherical harmonics $\Y^i_k\Y^j_k$ which is obviously  a subspace of $L_{2}(S^{2}\times S^{2})$ whose basis is  $\Y^i_k\Y^j_l$.

Because $S^3$ is a double covering of $SO(3)$ there is also a close connection to the spherical Radon transform  of even functions considered by S. Helgason (\cite{helgason}, \cite{helgason1}) as a special case of Radon tranforms on homogeneous spaces. A comparison of $\RD$ to the spherical Radon transform of even functions is given in \cite{bernstein/schaeben}.

The numerical aspect of solving the crystallographic problem as well as the inversion of $\RD$ is treated in \cite{BPS} and \cite{hielscher}.

To tackle the real-world problem of sharp textures, i.e. the pole figures consists mainly of a few delta peaks, wavelets on the sphere $S^3$ and the group $SO(3)$ and their behavior under $\RD$ had been studied in \cite{MMAS_I} and \cite{Ebert}. In \cite{diff-wavelets} and \cite{Ebert} an grouptheoretical approach is used to generalize the results for wavelets on groups and homogeneous spaces which leads to a new class of wavelets, called diffusive wavelets.

Another approach to solve the inversion problem of $\RD$ is done in \cite{CFKT}, where Gabor frames on the $Spin(3)$ had been used to solve the problem numerically.

In section 2 of the paper we use a group theoretic and representation theoretic approach to describe the Radon transform on a general compact Lie group $\G.$
At first hand $(\RD f)(x,y)$ seems to be defined over $\G\times \G$ while deeper inverstigations reveal that $\RD f$ is invariant under right shifts of $x$ as well as $y$ and hence $\RD$ is rather defined over $\G \slash \h \times \G \slash \h .$ For the practical application $\G = SO(3)$ and $\h = SO(2)$ and thus $\G \slash \h = SO(3)\slash SO(2) = S^2.$
In this particular situation a number of explicit inversion formulas is known. For more details we refer to \cite{asg}, \cite{bernstein/schaeben}, \cite{BHS}, \cite{bunge}, \cite{HPPSS},  \cite{K}, \cite{Matthies}- \cite{P}.

In  section 4 of the paper we  discuss generalized  variational splines on compact Riemannian manifolds  (see \cite{Pes2000}, \cite {Pes2004}, \cite{Pes2004-1}) and apply this concept to the Radon transform.

In \cite{Pes2000}-\cite{Pes2004} a theory of generalized (=average) interpolating variational splines was developed in compact and non-compact manifolds of bounded geometry. We use term generalized splines to stress that we interpolate functions on manifolds  by using values of their integrals over submanifolds  of a given manifold. It includes  the classical way of interpolating functions is by using their values on discrete sets of points.  Variational splines on manifolds  were used in our papers for reconstruction of Paley-Wiener functions on manifolds. In \cite{Pes2004} the  theory of generalized variational splines was applied to the spherical Radon transform of even functions (Funk transform) and to the hemispherical transform of odd functions on $n$-dimensional spheres.

The development in the present paper is slightly different from our approach in previous papers  in the sense that here we introduce splines using general elliptic second order self-adjoint operators  and not only  the Laplace-Beltrami operator.
 We prove  existence and uniqueness of interpolating variational
 splines associated with a  general elliptic second order self-adjoint operator on a compact Riemannian manifold. We remind (see Theorem \ref{FundSol}) that every generalized  spline is a linear combinations of what we call generalized fundamental solutions (= generalized Green functions). Using these results "explicit" formulas of
  variational splines in terms of  eigenfunctions of the corresponding  operator are given.
    
    In section 3 we develop our first method for approximate inversion of Radon transform on compact Lie groups. Note that in a similar situation for Radon a hemispherical transforms on unit spheres this idea was  introduced in  \cite{Pes2004}.  The idea is that one can always have an approximate inversion of "any kind" of Radon transform by treating inversion as a generalized interpolation problem. Thus, we invert group Radon transform by "interpolating" a function of the group using its integrals over a finite family of appropriate submanifolds.

    In the last section we review our fundamental inequalities (Lemma \ref{zerolem}, Theorem \ref{zeroThm}) for functions with many zeroes (see \cite{Pes2000}, \cite {Pes2004}, \cite{Pes2004-1}) and our Approximation Theorem for interpolating variational splines on general compact manifolds (\cite{Pes2004}, \cite{Pes2004-1}).  The rate of approximation in this theorem is expressed in terms of Sobolev norms. It is worth to note that using standard methods of interpolation theory and the fact that Besov spaces on manifolds are interpolation spaces (real interpolation method) between two Sobolev spaces the corresponding  inequalities can be extended to Besov norms.

    We use these results to describe our second method of approximate inversion of the Radon transform on $SO(3)$. The idea of this method is to interpolate not the function but its Radon transform on the manifold $S^{2}\times S^{2}$ and then "to return" to the original manifold $SO(3)$. This method allows to obtain explicit estimates of degree of approximation in Theorem 6.5.  Moreover, in the same section we establish a Sampling Theorem for Radon transform on $SO(3)$ (Theorem 6.6)   which says that 
one can have a complete reconstruction of $\omega$-bandlimited (on $SO(3)$) function $f$  by using only the values of its integrals over sufficiently dense specific  family of   submanifolds of $SO(3)$. 
    
It should be stressed that the variational spline problem fits in some sense optimal to the
practical question of determining the ODF $f$ from measurements of
the Radon transformed $f$. On the one hand we are interested in
regions where the values of $f$ are large but if the curvature of
$f$ is small in those regions it would be more useful to increase
the measurements around the maximum of $f$ and those points where
the curvature is large. Hence the right criteria to increase the
density of measurements around points where $(1-\Delta)f$ is large. 
The density of measurements should be high at those points where the
interpolation is highly nonlinear.

It should be noted that a very different and  more constructive approach to splines on two-dimensional surfaces was developed in \cite{ANS}, \cite{DS}, \cite{F}-\cite{FS98}, \cite{LS}.

An approach to splines on compact manifolds which is similar to our approach in   \cite{Pes2000}-\cite{Pes2004} (in the sense it is based on a  "zeroes lemma" and Green's function representations) was recently developed in \cite{HNW}, \cite{HNSW}  for the classical way of interpolation on discrete sets of points.

\section{Fourier analysis and Radon transform on compact Lie groups}

\subsection{Fourier Analysis on compact groups}

One of the most important theorems of functional analysis is the spectral theorem for compact self-adjoint operators on a Hilbert space. Which states that if
$A$ is a compact self-adjoint operator on a Hilbert space $V$, then there is an orthonormal basis of $V$ consisting of eigenvectors of $A$ and each eigenvalue is real. The analog of this theorem is the Peter-Weyl theorem. We want to recollect some basic notations and properties. 

Let $\G$ be a compact Lie group. A unitary representation of $\G$ is a continuous group homomorphism $\pi$: $\G\to U(d_{\pi})$ of $\G$ into the goup of unitary matrices of a certain dimension $d_{\pi}$ which will be explained later in the Peter-Weyl theorem. Such a representation is irreducible if $\pi(g)M=M\pi(g)$ for all $g\in\G$ and some $M\in \mathbb{C}^{d_{\pi}\times d_{\pi}}$ implies $M=cI$ is a multiple of the identity. Equivalently, $\mathbb{C}^{d_{\pi}}$ does not have non-trivial $\pi$-invariant subspaces $V\subset \mathbb{C}^{d_{\pi}}$ with $\pi(g)V \subset V$ for all $g\in\G.$ Two representations $\pi_1$ and $\pi_2$ are equivalent, if there exists an invertible matrix $M$ such that $\pi_1(g)M=M\pi_2(g)$ for all $g\in\G$.

Let $\hat{\G}$ denote the set of all equivalence classes of irreducible representations. Then this set parametrerizes an orthogonal decomposition of $L^2(\G).$
\begin{thm}[Peter-Weyl, \cite{vilenkin1}] Let $\G$ be a compact Lie group. Then the following statements are true.
\begin{description}
	\item[a] Denote $\H_{\pi}=\{g\mapsto {\rm trace}(\pi(g)M): M\in \mathbb{C}^{d_{\pi}\times d_{\pi}}\}.$ Then the Hilbert space $L^2(\G)$ decomposes into the orthogonal direct sum
	\begin{eqnarray}
		L^2(\G) = \bigoplus_{\pi\in \hat{\G}} \H_{\pi}
	\end{eqnarray}
	\item[b] For each irreducible representation $\pi\in \hat{\G}$ the orthogonal projection  \\ $L^2(\G)\to \mathcal{H}_{\pi}$ is given by
	\begin{eqnarray}
		f \mapsto d_{\pi} \int_{\G} f(h)\chi_{\pi}(h^{-1}g)\,dh = d_{\pi}\,f*\chi_{\pi},
	\end{eqnarray}
	in terms of the character $\chi_{\pi}(g)={\rm trace}(\pi(g))$ of the representation and $dh$ is the normalized Haar measure.
\end{description}
\end{thm}
We will denote the matrix $M$ in the equation $f*\chi_{\pi}={\rm trace}(\pi(g)M)$ as Fourier coefficient $\hat{f}(\pi)$ of $f$ at the irreducible representation $\pi$. The Fourier coefficient can be calculated as
$$ \hat{f}(\pi) = \int_{\G} f(g)\pi^*(g)\,dg .$$
The inversion formula (the Fourier expansion) is then given by
$$ f(g) = \sum_{\pi\in\hat{G}} d_{\pi}\,{\rm trace}(\pi(g)\hat{f}(\pi)). $$
If we denote by $||M||^2_{HS}={\rm trace}(M^*M)$ the Frobenius or Hilbert-Schmidt norm of a matrix $M,$ then the following Parseval identity is true.
\begin{col}[Parseval identity] Let $f\in L^2(\G).$ Then the matrix-valued Fourier coefficients $\hat{f}\in \mathbb{C}^{d_{\pi}\times d_{\pi}}$ satisfy
\begin{eqnarray}
||f||^2 = \sum_{\pi\in\hat{\G}} d_{\pi}\,||f(\pi)||^2_{HS} . \label{parseval_id}
\end{eqnarray}
\end{col}
On the group $\G$ one defines the convolution of two integrable functions $f,\,r\in L^1(\G)$ as
$$ f*r(g) = \int_{\G} f(h)r(h^{-1}g)\,dh . $$
Since $f*r\in L^1(\G),$ the Fourier coefficients are well-defined and they satisfy
\begin{col}[Convolution theorem on $\G$] Let $f,\,r\in L^1(\G)$ then $f*r\in L^1(\G)$ and
	$$ \widehat{f*r}(\pi) = \hat{f}(\pi)\hat{r}(\pi). $$
\end{col}
The group structure gives rise to left and right translations $T_gf\mapsto f(g^{-1}\cdot)$ and $T^gf\mapsto f(\cdot g)$ of functions on the group. A simple computation shows
$$ \widehat{T_gf}(\pi) = \hat{f}(\pi)\pi^*(g) \quad \mbox{and}\quad \widehat{T^gf}(\pi) = \pi(g)\hat{f}(\pi). $$
They are direct consequences of the definition of the Fourier transform. 

The Laplace-Beltrami operator $\Delta_{\G}$ on the group $\G$ is bi-invariant, i.e. commutes with all $T_g$ and $T^g.$ Therefore, all its eigenspaces are bi-invariant subspaces of $L^2(\G).$ As $\mathcal{H}_{\pi}$ are minimal bi-invariant subspaces, each of them has to be eigenspace of $\Delta_{\G}$ with corresponding eigenvalue $-\lambda_{\pi}^2.$ Hence, we obtain
$$ \Delta_{\G}f = -\sum_{\pi\in\hat{G}} d_{\pi}\,\lambda_{\pi}^2\,{\rm trace}(\pi(g)\hat{f}(\pi)). $$

 Motivated by situation which was described  we introduce a new type of Radon transform which is an abstract version of the crystallographic Radon transform.

  


\subsection{Radon transform on compact groups}

Now, we are able to define the Radon transform.

\begin{defn}
    Let $\h$ be a closed subgroup of the compact Lie group $\G$. The Radon transform of a continuous function $f\in C(\G)$ is defined by
    \begin{align}\label{def:RD}
        \RD f(x,y) = \int_\h f(xhy^{-1}) \dx h \qquad x,y\in \G,
    \end{align}
where $dh$ here is the normalized Haar measure on $\h$.
\end{defn}

Next, we explain that $\RD$ maps $\G$ into $\G/\h \times \G/\h.$ For that we use the averaging method.

For the following discussion we mention, that functions on $\G/\h$ can be regarded as functions on $\G$, which are constant over right co-sets of the form $g\h=\{gh,\,h\in\h\}$ for all $g\in \G$. The projection of a function on $\G$ onto functions on  $\G/\h$ corresponds to an averaging method over $g\h$:
\begin{align}
    \P_\h f(g) = \int_\h f(gh)\dx h.
\end{align}
It can be shown, that $\P_\h$ in Fourier domain acts by multiplying the Fourier coefficients $\hat f(\pi)$ by
\begin{align}\label{eq:proj_matrix}
    \pi_\h=\int_\h \pi(h)\dx h
\end{align}
from the right. Further $\pi_\h$ is a projection  and without loss of generality $\pi_\h=diag(1,...,1,0,..,0)$, where the number of $1$'s corresponds to the number of $\h$ invariant vectors in the representation Hilbert space of $\pi$. For details on the projection see \cite{Vilenkin}.
Next we discuss the Range of $\RD$. Since $x,y$ in \eqref{def:RD} are elements of $\G$ at a first look it seems that the Radon transform is defined over $\G\times\G$. While a deeper investigation reveals that $\RD f(x,y)$ is invariant under right shifts of $x$ as well as under right shifts of $y$, hence $\RD$ is rather defined over $\G/\h\times\G/\h$.

\begin{lem}
The Radon transform $\RD$ maps functions over $\G$ to functions over $\G/\h\times\G/\h$. 
\end{lem}
\begin{proof}
We look at $\RD$ on the Fourier domain. Let first $y\in \G$ be fixed and regard $\RD f(\cdot,y)$ to be a function on $\G$ in the first argument, then
\begin{equation}
    \widehat{\RD f (\cdot,y)}(\pi) = \pi_{\h} \pi^{*}(y) \widehat {f}(\pi)  \pi\in\widehat {\G},\label{Fourier-coeff_RD_foxed_y}
\end{equation}
Hence the the function $\RD f(\cdot,y)$ is invariant under the projection $\P_{\h}$, since the Fourier coefficients are invariant under the left multiplication by $\pi_{\h}$: 
$$
\pi_{\h} \pi_{\h} \pi^{*}(y) \widehat {f}(\pi)=\pi_{\h} \pi^{*}(y) \widehat {f}(\pi).
$$ 
Consequently,
\begin{equation}
   \RD f (x\cdot h,y) =\RD f (x,y)   \>\>h\in \h.
\end{equation}
A look at the Radon transform as function in the second argument $y$, while the first one $x=x_{0}$ is fixed, we find
$$
    \P_{\h} \RD f (x_{0},y) = \int_{\h}\RD f(x_{0},yh)\> dx h =
    $$
    $$
    \int_{\h}\sum_{\pi\in\widehat {\G}} d_{\pi} (\widehat{ f}(\pi) \pi(x_{0}))\pi_{\h} \pi(h^{-1}y^{-1}) \> dx h
    $$
    $$
    = \sum_{\pi\in\widehat {\G}} d_{\pi}(\widehat {f}(\pi) \pi(x_{0}))\pi_{\h} \pi^{*}(y) = \RD f(x_{0},y),\label{RD_fourier_exp}
$$
Hence $\RD f(x,y)$ is constant over fibers of the form $y\h$ also in the second argument and
\begin{equation}
    \widehat{\RD f(x,\cdot)}(\pi) = \pi_{\h}\overline{\pi^{*}(x)\widehat {f}(\pi)^{*}},
\end{equation}
where the  complex conjugate of the matrices is taken componentwise.
Consequently $\RD$ maps functions over $\G$ to functions over $\G/\h\times\G/\h$. 
\end{proof}

Below we  discuss the Radon transform of the space $L^2(\G)$, its range and its inversion.

\begin{lem}
    Let $\h$ be the subgroup of $\G$, determining the Radon transform on $\G$ and let $\widehat \G_1\subset\widehat \G$ be the set of irreducible representations with respect  to $\h$. Then for $f\in C^\infty(\G)$ it is
    \begin{align}\label{lemma:RD_parseval}
        \|\RD f\|^2_{L^2(\G/\h\times\G/\h)} &= \sum_{ \pi \in \widehat\G_1}\mbox{rank}(\pi_\h) \|\widehat f(\pi)\|^2_{HS}.
    \end{align}
\end{lem}

\begin{proof}
    For the proof we expand $\RD f(x,y)$ for fixed $y$ as function in $x$ over $\G$ (or better $\G/\h$) and apply Parseval's identity \eqref{parseval_id}, with \eqref{Fourier-coeff_RD_foxed_y} we have
    \begin{align*}
        \|\RD f \|^2_{L^2(\G/\h\times \G/\h)} &= \sum_{\pi\in\widehat\G} d_\pi \int_\G\|\pi_\h\pi^*(y)\widehat f(\pi)\|_{HS}^2\dx y\\
        &= \sum_{\pi\in\widehat\G} d_\pi \int_\G\tr\left(\widehat f^*(\pi)\pi(y)\pi_\h\pi^*(y)\widehat f(\pi)\right)\dx y\\
        &=\sum_{\pi\in\widehat\G} d_\pi \tr\left(\widehat f^*(\pi) \int_\G \pi(y)\pi_\h\pi^*(y)\dx y \,\widehat f(\pi)\right)\\
        &= \sum_{\pi\in\widehat\G_1} \mbox{rank}(\pi_\h)\tr(\widehat f^*\widehat f) = \sum_{\pi\in\widehat\G_1} \mbox{rank}(\pi_\h) \|\widehat f(\pi)\|_{HS}^2.
    \end{align*}
    Where we made use of
    \begin{align}
        \int_\G \widehat \pi(y)\pi_\h\pi^*(y)\dx y &= \left(\sum_{k=1}^{\mbox{rank}\pi_\h}\int_\G \pi_{ik}(y)\overline{\pi_{kj}(y)}\dx y\right)_{i,j=1}^{d_\pi} \nonumber\\
        &= \frac{\mbox{rank}(\pi_\h)}{d_\pi} Id. \label{hilfsrechnung}
    \end{align}
\end{proof}

\section{Harmonic analysis and Radon transform on $SO(3)$}

\subsection{Special functions of $SO(3)$}

For the Hilbert space $L^2(S^2)$ we use the orthonormal system of spherical harmonics in the $\{\Y_k^i,\,i=1,...,2k+1,k\in \N_0\}$.

In case of $SO(3)$ the situation is very comfortable, since every irreducible representation is unitary equivalent to a irreducible component of the quasi regular representation in $L^2(S^2)$, given by
\begin{align}
    T(g):f(\xi)\mapsto f(g^{-1}\cdot x),
\end{align}
where $\cdot$ denotes the canonical action of $SO(3)$ on $S^2$. The irreducible invariant components of $L^2(S^2)$ under $T$ are $\H_k=\{\Y_k^i,i=1,...,2k+1\}$- spanned by spherical harmonics of degree $k$. $T^k$ shall denote the irreducible representation, obtained by restriction of $T$ to $\H_k$.
The matrix coefficients of $T^k$ are the Wigner polynomials $T_{ij}^k$ of degree $k$:
\begin{align}
    \Y_k^j(g^{-1}\cdot \xi) &= \sum_{i=1}^{2k+1} T^k_{ij}(g)\Y_k^i(\xi)& T_{ij}^k(g)&=\langle \Y_k^j(g^{-1}\cdot),\Y_k^i(\cdot) \rangle_{L^2(S^2)}.
\end{align}
By construction the projection matrix on the Fourier handside given in \eqref{eq:proj_matrix} is $\pi_{SO(2)}(k)=diag(1,0,...,0)$ where number of zeros is $2k$ and $\pi_{SO(2)}(k)$ of dimension $(2k+1)\times(2k+1)$. Since matrix coefficients always have the norm $\frac{1}{d_\pi}$, where $d_\pi$ is the dimension of the representation, we have
\begin{align}
    T^k_{i1}(g) &= \sqrt{\frac{4\pi}{2k+1}}\Y_k^i(g\cdot\xi_0),
\end{align}
where $\xi_0\in S^2$ is the base point of $SO(3)/SO(2)\sim S^2$, often chosen as north pole and its stabilizer is the factorized subgroup $SO(2)$.

The eigenvalue of Laplacian on $SO(3)$ and on $S^2$ corresponding to polynomials of degree $k$ is $-k(k+1)$:
\begin{align}\label{eigenvalues}
    \Delta_{SO(3)} T^k_{ij}=-k(k+1) T^k_{ij} ,\>\>\> \Delta_{S^2}\Y_k^i=-k(k+1)\Y_k^i.
\end{align}

\begin{defn} The subgroup $\h$ is called massive, if ${\rm rank}\hat{\h}(\pi) \leq 1$
for all $\pi\in \hat{\G}.$ Furthermore, an irreducible representation $\pi\in\hat{\G}$ is called class-1 representation with respect to the subgroup $\h,$ if ${\rm rank}\hat{\h}(\pi) \geq 1.$
\end{defn}

\begin{lem}{\rm (\cite[Chapter IX.2.6]{Vilenkin})} $SO(n)$ is a massive subgroup of $SO(n+1).$ Furthermore, the family $T^k,\,k\in \mathbb{N}_0,$ gives up to equivalence all class-1 representations of $SO(n+1)$ with respect to $SO(n).$
\end{lem}

For the following we fix the 'north pole' $\xi_0$ of $S^2.$ Then the set of zonal spherical harmonics is one-dimensional and spanned by the Gegenbauer polynomials.
Further the dimension of zonal functions in $\H_k$ is one for all $k\geq 0$ and it is spanned by Gegenbauer polynomial of order $C^{\frac12}_k(\xi_0\cdot\xi)$, note that $\cos(\angle(\xi_0,\xi)=\xi_0\cdot\xi)$. Hence, to be zonal on $S^2$ for a function $f(\xi)$ means to be invariant under the action of $SO(2)$, i.e. to depend only on the angle between the argument $\xi$ and the $SO(2)$ invariant point $\xi_0$.


The following addition theorem holds true
\begin{thm}[Addition theorem]\label{addition_theorem}
For all $\xi,\eta\in\S^2$ and $k\in \N_0$
\begin{equation}
 \mathcal C_k^{\frac12} (\xi\cdot\eta)
  = \frac{4\pi}{2k+1} \sum_{i=1}^{2k+1} \Y^i_k(\xi)\overline{\Y^i_k(\eta)}.
\end{equation}
\end{thm}

\subsection{Radon transform on $SO(3)$}

The Radon transform on $SO(3)$ is defined in Definition \ref{def:RD_on_SO(3)}. We also need apropriate function spaces, which adjust the smoothness of the functions. Below we define Sobolev spaces for $S^2\times S^2$.

\begin{defn}\label{SobNorm} The Sobolev space $H_{t}(S^2\times S^2),\,t\in \mathbb{R},$ is defined as the domain of the operator $(1-2\Delta_{S^2\times S^2})^{\tfrac{t}{2}}$ with graph norm
	$$ ||f||_t = ||(1-2\Delta_{S^2\times S^2})^{\tfrac{t}{2}}f||_{L^2(S^2\times S^2)} ,\>\>f\in L^2(S^2\times S^2). 
	$$
\end{defn}
Since 
\begin{equation}\label{actiononT}
\RD(T^k)(x,y)=T^k(x)\pi_{SO(2)}(T^k(y))^*
\end{equation}
we have
\begin{align}\label{basis-action1}
    \RD T^k_{ij}(\xi,\eta) &= T_{i1}^k(\xi) \overline{T^k_{j1}(\eta)} =\frac{4\pi}{2k+1} \Y_k^i(\xi)\overline{\Y_k^j(\eta)}.
\end{align}This formula shows that range of $\RD$ belongs to kernel of the Darboux-type operator i.e.
\begin{equation}
\Delta_{x}\RD f(x,y)=\Delta_{y} \RD f(x,y),\>\>\>f\in L_{2}(SO(3)).
\end{equation}
Because $\RD f$ is in the kernel of the Darboux-type operator we also need the following Sobolev-type function space.
\begin{defn} The Sobolev space $H_t^{\Delta}(S^2 \times S^2),\,t\in \mathbb{R},$ is defined as the subspace of all functions $f\in H_t(S^2\times S^2)$ such $\Delta_1 f = \Delta_2 f.$ 
\end{defn}
Now  we define Sobolev spaces on $SO(3)$.
\begin{defn} The Sobolev space $H_t(SO(3)),\,t \in \mathbb{R},$ is defined as the domain of the operator $(1-4\Delta_{SO(3)})^{\tfrac{t}{2}}$ with graph norm
		$$ |||f|||_t = ||(1-4\Delta_{SO(3)})^{\tfrac{t}{2}}f||_{L^2(SO(3))},\>\>f\in L^2(SO(3)). $$
\end{defn}	
Because the operators $1-2\Delta_{S^2\times S^2}$ and $1-4\Delta_{SO(3)}$  have positive spectrum and one can choose corresponding eigenfunctions which form an orthonormal basis.
These definitions are consistent with another norm on the Sobolev space   $H_t(M)$ which will be given in Definition \ref{Sobolevnorm}.

\begin{thm}(Range description)\label{RD_SO(3)-thrm}
    For any $t\geq 0$ the Radon transform on $SO(3)$ is an invertible  mapping
        \begin{align}
            \RD: H_{t}(SO(3))\to H_{t+\frac{1}{2}}^{\Delta}(S^2\times S^2).
        \end{align}
\end{thm}
\begin{proof}
   It is sufficient to consider case $t=0$.  If  $d_k=2k+1$ is the dimension of the irreducible representations and $-\lambda_k^2=-k(k+1)$ are the eigenvalues of the Laplacian $\Delta_{SO(3)}$ we have $d_k=\sqrt{1+4\lambda_k^2}$.  Since $\h=SO(2)$ is massive in $\G = SO(3)$  and $T^k$ are class-1 representations of $SO(3)$ with respect to $SO(2)$ we have ${\rm rank}(\pi_{\h}) = 1$ and $\widehat \G_1 = \widehat \G $. Now the assertion follows from (\ref{lemma:RD_parseval}) and (\ref{hilfsrechnung}). We have
    \begin{align*}
    \|\RD f\|^2_{\frac{1}{2}} & = \|(1-2\Delta_{S^2\times S^2})^{\frac{1}{4}}\RD f\|^2_{L^2(S^2\times S^2)} = 	\sum_{\pi\in\widehat\G} d_\pi \int_\G\|\sqrt{d_{\pi}}\pi_\h\pi^*(y)\widehat f(\pi)\|_{HS}^2\dx y \nonumber \\
        &= \sum_{\pi\in\widehat\G} d^2_\pi \int_\G\tr\left(\widehat f^*(\pi)\pi(y)\pi_\h\pi^*(y)\widehat f(\pi)\right)\dx y \nonumber \\
       & =\sum_{\pi\in\widehat\G} d^2_\pi \tr\left(\widehat f^*(\pi) \int_\G \pi(y)\pi_\h\pi^*(y)\dx y \,\widehat f(\pi)\right) \nonumber \\
        &= \sum_{\pi\in\widehat\G_1=\widehat \G} d_{\pi}\mbox{rank}(\pi_\h)\tr(\widehat f^*\widehat f) = \sum_{\pi\in\widehat\G} d_{\pi}  \|\widehat f(\pi)\|_{HS}^2 = ||f||^2_{L^2(SO(3))} = |||f|||^2_0
    \end{align*}
\end{proof}

From Theorem \ref{RD_SO(3)-thrm} we deduce the reconstruction formula for the Radon transform on $SO(3)$. This is also the result of a simple calculation involving spherical harminoncs as well as Wigner polynomials. Using the identities (\ref{basis-action1}) we get the following theorem:

\begin{thm}(Reconstruction formula)

Let
\begin{align}
   (\RD g)(x,y) = f(x,y) &= \sum_{k=0}^\infty \sum_{i,j=1}^{2k+1} \widehat f(k,i,j) \Y_k^i(x) \overline{\Y_k^j(y)} \in H_{\frac{1}{2}}^{\Delta}(S^2\times S^2)
\end{align}
be the result of a Radon transform. Then the pre-image $g\in L^2(SO(3))$ is given by
\begin{align}
    g &= \sum_{k=0}^\infty \sum_{i,j=1}^{2k+1} \frac{(2k+1)}{4\pi} \widehat f(k,i,j) T_{ij}^k =\sum_{k=0}^\infty (2k+1)\tr(\widehat g(k) T^k)\\
    \widehat g(k,i,j) &=\frac{1}{4\pi}\widehat f(k,i,j).
\end{align}
\end{thm}

\section{Generalized variational splines  on compact Riemannian manifolds}

We consider a
compact Riemannian manifold $M$ without boundary.  Let $\mathcal{L}$ be a differential of order two elliptic operator which is  self-adjoint and negatively semi-definite  in the
space $L_{2}(M)$ constructed using a Riemannian density $dx$. The spectrum of such operator always contains $\lambda_{0}=0$. In order to have an invertible operator we will work with $I-\mathcal{L}$, where $I$ is the identity operator in $L_{2}(M)$.  It is known that for every such operator $\mathcal{L}$ the domain of the power $(-\mathcal{L})^{t/2},\>t\in \mathbb{R}$, is the Sobolev space $H_{t}(M)$. There are different ways to introduce norm in Sobolev spaces. We choose the following definition. 
\begin{defn}
\label{Sobolevnorm}
The Sobolev space $H_{t}(M), t\in \mathbb{R}$ can be introduced as the
domain of the operator $(1-\mathcal{L})^{t/2}$ with the graph norm

$$
\|f\|_{t}=\|(1-\mathcal{L})^{t/2}f\|, f\in H_{t}(M).
$$
\end{defn}

Note, that such norm depends on $\mathcal{L}$. However, for every two differential of order two elliptic operators  such norms are equivalent for each $t\in \mathbb{R}$. 

Since the operator $(-\mathcal{L})$ is  self-adjoint and positive semi-definite it has a discrete
spectrum $0=\lambda_{0}<\lambda_{1}\leq \lambda_{2}\leq...,$ and
one can choose corresponding eigen functions $\varphi_{0},
\varphi_{1},...$ which form an orthonormal basis of $L_{2}(M).$ A
distribution $f$ belongs to $H_{t}(M), t\in \mathbb{R},$ if and
only if
$$
\|f\|_{t}=
\left(\sum_{j=0}^{\infty}(1+\lambda_{j})^{t}|c_{j}(f)|^{2}\right)^{1/2}<\infty,
$$
where Fourier coefficients $c_{j}(f)$ of $f$ are given by
$$
c_{j}(f)=\left<f,\varphi_{j}\right>=\int_{M}f\overline{\varphi_{j}}.
$$

\bigskip

This $L_{2}-$inner product can be also considered as a pairing
between $H_{-t}(M)$ and $H_{t}(M)$ and in this sense every element
of $H_{-t}(M)$ can be identified with a continuous functional on
$H_{t}(M)$.

For a given finite family of pairwise different submanifolds $\{M_{\nu}\}_{1}^{N}$ consider the following family of  distributions
 \begin{equation}
 \label{functionals}
 F_{\nu}(f)=\int_{M_{\nu}}f
 \end{equation}
 which are well defined at least for functions in $H_{\varepsilon +d/2}(M),\>\>\varepsilon>0$.
 In particular, if $M_{\nu}=x_{\nu}\in M$, then every $F_{\nu}$ is a Dirac measure $\delta_{x_{\nu}}\>\>\nu=1,...,N,\>\>x_{\nu}\in M.$
 
 Note that distributions $F_{\nu}$ belong to $H_{-\varepsilon-d/2}(M)$ for any $\varepsilon>0$.
 
\begin{Vario}
 Given a sequence of complex numbers
$v=\{v_{\nu}\},$ \newline $ \nu=1,2,...,N,$ and a $t>d/2$ we consider the
following variational problem:

\bigskip

\textsl{Find a function  $u$ from the space $H_{t}(M)$ which has
the following properties:} \label{var_prob}

\begin{enumerate}

\item $ F_{\nu}(u)=v_{\nu}, \nu=1,2,...,N, v=\{v_{\nu}\},$

\item  $u$ \textsl{minimizes functional $u\rightarrow \|(1-\mathcal{L})
^{t/2}u\|$.}

\end{enumerate}
\end{Vario}
\bigskip

We show that the solution to Variational problem exist and is
unique for any $t>t_{0}$ .  We need
the following  Independence  Assumption in order to determine the Fourier
coefficients of the solution.

\bigskip

 \textbf{Independence Assumption.} \textsl{There are functions
  $\vartheta_{\nu}\in C^{\infty}(M)$ such
that}

\begin{equation}
\label{independence}
F_{\nu}(\vartheta_{\mu})=\delta_{\nu\mu},
\end{equation}
\textsl{where $\delta_{\nu\mu}$ is the Kronecker delta.}

Note, that this assumption implies in particular that the
functionals $F_{\nu}$ are linearly independent. Indeed, if we have
that for certain coefficients $\gamma_{1},\gamma_{2}, ...,
\gamma_{N}$
$$
\sum _{\nu=1}^{N}\gamma_{\nu}F_{\nu}=0,
$$
then for any $1\leq\mu\leq N$
$$
0=\sum_{\nu=1}^{N}\gamma_{\nu}F_{\nu}(\vartheta_{\mu})=\gamma_{\mu}.
$$

The families of distributions that satisfy our condition
  include

  a) Finite families of $\delta$ functionals and their
  derivatives.

  b) Sets of integrals over submanifolds from a finite family of
submanifolds of any codimension.

 The solution to the Variational
Problem will be called a spline and will be denoted as $s_{t}(v).$
 The set of all solutions for a fixed set of distributions
$F=\{F_{\nu}\}$ and a fixed $t$ will be denoted as $S(F,t).$

\begin{defn}
\label{interpolationdef}
Given a function $f\in H_{t}(M)$ we will say that the unique
spline $s$ from $S(F,t)$ interpolates $f$  if
$$
F_{\nu}(f)=F_{\nu}(s).
$$
Such spline will be denoted as $s_{t}(f).$
\end{defn}

From the point of view of the classical theory of variational
 splines it would be more natural to consider minimization of the
functional
$$
u\rightarrow \|\mathcal{L}^{t/2}u\|.
$$
However, in the case of a general compact manifolds it is easer to
work with the operator $1-\mathcal{L}$ since this operator is
invertible. The following existence and uniqueness theorem was proved in \cite{Pes2004}.

\begin{thm}
\label{UETh}The \textsl{
Variational Problem} has a unique solution for any
sequence of values $(v_{1},v_{2},... v_{N})$.
\end{thm}

\begin{proof}
Consider the set 
$$
\bigcap_{\nu} Ker \>F_{\nu}=V^{0}_{t}(F)\subset H_{t}(M), t>d/2, 
$$ of all
functions in $H_{t}(M)$ such that for every $1\leq \nu\leq N,
F_{\nu}(f)=0.$

Given a sequence of complex numbers $(v_{1}, v_{2}, ..., v_{N})$ the
 linear manifold

 $$
 V_{t}(F,v_{1},...v_{N}), t>d/2
 $$
 of all functions $f$ in
$H_{t}(M)$ such that $F_{\nu}(f)=v_{\nu}, \nu=1,...,N,$ is a
shift of the closed subspace $V^{0}_{t}(F)$, i.e.

$$ V_{t}(F,v_{1},...,v_{N})=V^{0}_{t}(F)+g,$$
where $g$ is any function from $H_{t}(M)$ such that
$F_{\nu}(g)=v_{\nu}, \nu=1,2,...,N.$

Consider the orthogonal projection $g_{0}$ of $g\in H_{t}(M)$
onto the space $V^{0}_{t}(F)$ with respect to the inner product
in $H_{t}(M)$:

$$\left<f_{1},
f_{2}\right>_{H_{t}(M)}=\left<(1-\mathcal{L})^{t/2}f_{1},(1-\mathcal{L})^{t/2}f_{2}\right>_{L_{2}(M)}=
$$
$$
 \int_{M}(1-\mathcal{L})^{t/2}f_{1}\overline{(1-\mathcal{L})^{t/2}f_{2}}.
$$
Note,  that $s_{t}(v)=g-g_{0}\in V_{t}(F,v_{1},...,v_{N})$ is
the unique solution of the Variational Problem. Indeed, to show
that $s_{t}(v)$ minimizes the functional

$$u\rightarrow \|(1-\mathcal{L})^{t/2}u\|$$
on the set $V_{t}(F,v_{1},...,v_{N})$ we note that any function
in $V_{t}(F,v_{1},...,v_{N})$ can be written in the form
$s_{t}(v)+h,$ where $h\in V^{0}_{t}(F)$. For such a function we
have

$$
\|(1-\mathcal{L})^{t/2}(s_{t}(v)+h)\|^{2}=
$$
$$
\|(1-\mathcal{L})^{t/2}s_{t}(v)\|^{2}+2\left<s_{t}(v),h\right>_{H_{t}(M)}+
\|(1-\mathcal{L})^{t/2}h\|^{2}.
$$
Since $s_{t}(v)=g-g_{0}$ is orthogonal to $V^{0}_{t}(F)$ we obtain

$$
\|(1-\mathcal{L})^{t/2}(s_{t}(v)+\sigma h)\|^{2}=
\|(1-\mathcal{L})^{t/2}s_{t}(v)\|^{2}+
|\sigma|^{2}\|(1-\mathcal{L})^{t/2}h\|^{2}, \>\>h\in V^{0}_{t}(F),
$$
that shows that the function $s_{t}(v)$ is the minimizer.
\end{proof}
The following criterion follows from the previous theorem  \cite{Pes2004}.
\begin{thm}
A function $u\in H_{t}(M)$ is a solution of the Variational
Problem if and only if it is orthogonal to the subspace
$V^{0}_{t}(F)$ and $F_{\nu}(u)=v_{\nu}, \nu=1,2,... .$
\end{thm}

As a consequence  we obtain  that the set
of all solutions of  the Variational Problem is linear. In
particular, every spline $s_{t}(v)\in S(F,t)$ has the following
representation through its values $F_{\nu}(s_{t}(v))=v_{\nu},
\nu=1,...,N,$
 on $X$:
\begin{equation}
\label{representation-1}
s_{t}(v)=\sum_{\nu=1}^{N}v_{\nu}l^{\nu},
\end{equation}
where $F_{\nu}(s_{t}(v))=v_{\nu},$ and $l^{\nu}\in S(F,t),
\nu=1,2,... ,N, $ is so called Lagrangian spline that defined by
conditions $F_{\mu}(l^{\nu})=\delta_{\nu\mu}, \mu=1,2,...,N.$

The next Theorem gives the characteristic property of splines.

\begin{thm}
\label{MainTh-1}A
function $s_{t}(v)\in H_{t}(M),t>d/2$ is a solution of
 the Variational Problem  if and only if it satisfies the
following equation in the sense of distributions
\begin{equation}
\label{MainEquation}
(1-\mathcal{L})^{t}s_{t}(v)=\sum_{\nu=1}^{N}
\alpha_{\nu}(s_{t}(v))\overline{F_{\nu}}.
\end{equation}
In other words, for any  smooth $\psi$

$$\left<(1-\mathcal{L})^{t}s_{t}(v),\psi\right>_{L_{2}(M)}=\int_{M}(1-\mathcal{L})^{t}s_{t}(v)\overline{\psi} =\sum_{\nu=1}^{N}
\alpha_{\nu}(s_{t}(v))\overline{F_{\nu}(\psi )}.$$

\end{thm}
\begin{proof}
We already know that every solution of  the Variational
Problem is orthogonal to $V^{0}_{t}(F)$ in the Hilbert
space $H_{t}(M)$ i.e. for any $h\in V^{0}_{t}(F)$
\begin{equation}
\label{MainEquation-0}
0=\left<s_{t}(v),h\right>_{H_{t}(M)}
=\int_{M}(1-\mathcal{L})^{t/2}s_{t}(v)\overline{(1-\mathcal{L})^{t/2}h}.
\end{equation}

Note, that since we consider  a {\bf finite}  family of pairwise different manifolds $M_{\nu}$ our {\bf Independence Assumption} (\ref{independence})  is satisfied: there are functions
  $\vartheta_{\nu}\in C^{\infty}(M)$ such
that

\begin{equation}
\label{IndepAssump}
F_{\nu}(\vartheta_{\mu})=\delta_{\nu\mu},
\end{equation}
where $\delta_{\nu\mu}$ is the Kronecker delta.
Thus,  for any $\psi\in
C_{0}^{\infty}(M)$ the function

$$\psi-\sum_{\nu=1}^{N}F_{\nu}(\psi)\vartheta_{\nu}
$$
belongs to $V^{0}_{t}(F)$ and because of (\ref{MainEquation-0})

\begin{eqnarray*}
0=\left<s_{t}(v),\psi-\sum_{\nu=1}^{N}F_{\nu}(\psi)\vartheta_{\nu}\right>_{H_{t}(M)}  =
\int_{M}(1-\mathcal{L})^{t/2}s_{t}(v)
\overline{(1-\mathcal{L})^{t/2}(\psi-\sum_{\nu=1}^{N}F_{\nu}(\psi)\vartheta_{\nu})} \\
=\int_{M} (1-\mathcal{L})^{t}s_{t}(v)\left(\overline{\psi-
\sum_{\nu=1}^{N}F_{\nu}(\psi)\vartheta_{\nu}}\right).
\end{eqnarray*}
In other words,

$$
\int_{M}\left[(1-\mathcal{L})^{t}s_{t}(v)\right]\overline{\psi}=\sum
_{\nu=1}^{N}\overline{F_{\nu}(\psi)}\int_{M}(1-\mathcal{L})
^{t}s_{t}(v)\overline{\vartheta_{\nu}}.
$$
If we set
$$
\alpha_{\nu}(s_{t}(v),\vartheta)=\int_{M}\left[(1-\mathcal{L})^{t}s_{t}(v)\right]\overline{\vartheta_{\nu}},
$$
we obtain  that $(1-\mathcal{L})^{t}s_{t}(v)$ is a distribution of the
form

$$
(1-\mathcal{L})^{t}s_{t}(v)=\sum_{\nu=1}^{N}\alpha_{\nu}(s_{t}(v),\vartheta)
\overline{F_{\nu}},
$$
where
$$
\overline{F_{\nu}}(\psi)=\overline{F_{\nu}(\psi)}.
$$
 So
every solution of the variational problem is a solution of (\ref{MainEquation}).

Note, that if $\zeta=\{\zeta_{\nu}\}$ is another
$C_{0}^{\infty}(M)-$family for which
$F_{\nu}(\zeta_{\mu})=\delta_{\nu\mu}$, then we have the identity
$$
\sum_{\nu=1}^{N}(\alpha_{\nu}(s_{t}(v),\vartheta)-\alpha_{\nu}(s_{t}(v),\zeta))F_{\nu}=0.
$$
Since distributions $F_{\nu}$ are linearly independent it implies  that
$$
\alpha_{\nu}(s_{t}(v),\vartheta)-\alpha_{\nu}(s_{t}(v),\zeta)=0.
$$
In other words, coefficients 
$\alpha_{\nu}(s_{t}(v),\vartheta)=\alpha_{\nu}(s_{t}(v),\zeta)=\alpha_{\nu}(s_{t}(v))$ are
independent of the choice of the family of functions $\vartheta$.

Conversely, if $u$ is a solution of (\ref{MainEquation}) then since $F_{\nu}$
belongs to the space $H_{-t_{0}}(M),$ and $t>t_{0}\geq 0,$ the
Regularity Theorem for elliptic operator $(1-\mathcal{L})^{t}$  implies
that $u\in H_{-t_{0}+ 2t}(M)\subset H_{t}(M)$
 and for any $h\in V^{0}_{t}(F)$

\begin{eqnarray*}
\left<u,h\right>_{H_{t}(M)}=\left<(1-\mathcal{L})^{t/2}u,(1-\mathcal{L})^{t/2}h\right>=\left<(1-\mathcal{L})^{t}u,h\right>  \\ =
\sum_{\nu=1}^{N}\alpha_{\nu}(u)F_{\nu}(h)=0,
\end{eqnarray*}
that shows that $u$ is a the solution of  the Variational
Problem.
\end{proof}
The formula (\ref{representation-1}) represents spline through its values and Lagrangian splines. To obtain another representation of splines we will need what we call generalized fundamental solutions or generalized Green's functions 
$E_{\nu}^{t}, $  which are solutions of the following distributional equations
\begin{equation}
(1-\mathcal{L})^{t}E_{\nu}^{t}=\overline{F_{\nu}}.
\end{equation}
To find $E_{\nu}^{t}$ we note that in the sense of distributions

\begin{equation}
\overline{F_{\nu}}=\sum_{j=0}^{\infty}
\overline{F_{\nu}(\varphi_{j})}\varphi_{j}
\end{equation}
which shows that
\begin{equation}
E_{\nu}^{t}=\sum_{j=0}^{\infty}(1+\lambda_{j})^{-t}
 \overline{F_{\nu}(\varphi_{j})}\varphi_{j}.
\end{equation}

The following important fact holds.
\begin{thm}
\label{FundSol}
Every spline $s_{t}(v)$ is a linear combination of the generalized fundamental solutions
\begin{equation}
s_{t}(v)=\sum_{\nu=1}^{N}\alpha_{\nu}(s_{t}(v))E_{\nu}^{t}
\end{equation}
where scalar coefficients $\alpha_{\nu}(s_{t}(v))$ are the same as in Theorem \ref{MainTh-1}.
\end{thm}
\begin{proof}
According to  Theorem \ref{MainTh-1} every spline $s_{t}(v)$ is a solution
of

$$
(1-\mathcal{L})^{t}s_{t}(v)=
\sum_{\nu=1}^{N}\alpha_{\nu}(s_{t}(v))\overline{F_{\nu}}.
$$
Thus, we obtain the equality
$$
(1-\mathcal{L})^{t}s_{t}(v)=\sum _{\nu}\alpha_{\nu}(s_{t}(v))
 \sum_{j}\overline{F_{\nu}(\varphi_{j})}\varphi_{j},
$$
that implies the desired  representation.
\end{proof}
Note that so far we have used just the assumption that $t>d/2$.
To get more information about $s_{t}(v)$ we will need a stronger
assumption that $t>d$.

Theorems \ref{UETh} and \ref{FundSol} imply 
our main result concerning  variational splines (see \cite{Pes2004}).
\begin{thm}
\label{MainTheorem}
If $t>d$, then for any given sequence of scalars 
$v=\{v_{\nu} \}, \nu=1,2,...N,$ the following statements are
equivalent:

1) $s_{t}(v)$ is the solution to \textsl{the Variational Problem};

2) $s_{t}(v)$ satisfies the  equation (\ref{MainEquation}) in the sense of
distributions
\begin{equation}
\label{Main Equation-1}
(1-\mathcal{L})^{t}s_{t}(v)=\sum_{\nu=1}^{N}\alpha_{\nu}(s_{t}(v))\overline{F_{\nu}},
t>d,
\end{equation}
where $\alpha_{1}(s_{t}(v)),...,\alpha_{N}(s_{t}(v))$ form a
solution of the $N\times N$ system
\begin{equation}
\label{eq:linsystem}
\sum_{\nu=1}^{N}\beta_{\nu\mu}\alpha_{\nu}(s_{t}(v))=v_{\mu},
\mu=1,...,N,
\end{equation}
and
\begin{equation}
\label{eq:linsolution}
\beta_{\nu\mu}=\sum_{j=0}^{\infty}(1+\lambda_{j})^{-t}\overline{F_{\nu}(\varphi_{j})}
F_{\mu}(\varphi_{j}),\>\>\>\mathcal{L}\varphi_{j}=-\lambda\varphi_{j};
\end{equation}

3) the Fourier series of $s_{t}(v)$  has  the following form
\begin{equation}
\label{Fourier Series}
s_{t}(v)=\sum_{j=0}^{\infty}c_{j}(s_{t}(v))\varphi_{j},
\end{equation}
where
$$
c_{j}(s_{t}(v))=\left<s_{t}(v),\varphi_{j}\right>=(1+\lambda_{j})^{-t}
\sum_{\nu=1}^{N}\alpha_{\nu}(s_{t}(v))\overline{F_{\nu}(\varphi_{j})}.
$$
\end{thm}

\begin{remark}
It is important to note that the system (\ref{eq:linsystem}) is always solvable
according to our uniqueness and existence result for the
Variational Problem.

\end{remark}

\begin{remark}
It is also necessary to note that the series (\ref{eq:linsolution}) is absolutely
convergent if $t>d$. Indeed, since functionals $F_{\nu}$
are continuous on the Sobolev space $H_{d/2+\varepsilon}(M)$ we obtain that
for any normalized eigen function $\varphi_{j}$ which corresponds
to the eigen value $\lambda_{j}$ the following inequality holds
true

$$|F_{\nu}(\varphi_{j})|\leq
C(M,F)\|(1-\mathcal{L})^{d/4}\varphi_{j}\|\leq
C(M,F)(1+\lambda_{j})^{d/4},\>\>\>F=\{F_{\nu}\}.
$$
So
$$
|\overline{F_{\nu}(\varphi_{j})}F_{\mu}(\varphi_{j})| \leq
C(M,F)(1+\lambda_{j})^{d/2},
$$
and
$$
|(1+\lambda_{j})^{-t}\overline{F_{\nu}(\varphi_{j})}F_{\mu}(\varphi_{j})|\leq
C(M,F)(1+\lambda_{j})^{(t_{0}-t)}.
$$
It is known \cite{MP} that the series
$$
\sum_{j}\lambda_{j}^{-\tau},
$$
which defines the $\zeta-$function of an elliptic second order
operator, converges if $\tau>d/2$. This implies absolute
convergence of (\ref{eq:linsolution}) in the case $t>d$.

\end{remark}

One can show  that splines provide an optimal approximations to sufficiently smooth  functions.  Namely let $Q(F,f,t,K)$ be the set  of all functions $g$  in
$H_{t}(M)$ such that

\begin{enumerate}

\item  $F_{\nu}(g)=F_{\nu}(f), \nu=1,2,...,N,$

\item  $\|g\|_{t}\leq K,$ for a real $ K\geq \|s_{t}(f)\|_{t}.$

\end{enumerate}

The set $Q(F,f,t,K)$ is  convex,  bounded and closed. 

The following theorem (see \cite{Pes2004}) shows that splines provide an optimal approximations to functions in $Q(F,f,t,K)$.

\begin{thm}
\label{optim}
The spline $s_{t}(f)$ is the symmetry center of $Q(F,f,t,K)$. This means that for
any $g\in Q(F,f,t,K)$
\begin{equation}
\label{optim-ineq}
  \|s_{t}(f)-g\|_{t}\leq \frac{1}{2} diam \>Q(F,f,t,K).
\end{equation}
\end{thm}

\section{Approximate inversion of the group Radon transform using generalized variational interpolating splines }


We return to the Radon transform on a compact Lie group $\G$. In order to apply the general scheme developed in the previous section we select a finite number of pairs $(x_{\nu}, y_{\nu})\in \G\times \G, \nu=1,...,N,$ and introduce submanifolds $\M_\nu=x_\nu\h y_\nu^{-1}\subset \G.$ Note, that for any $\nu=1,...,N$ the dimension $dim\>\M_{\nu}=d_{\nu}$ equals $dim\> \h$. Next, for the set of scalars
\begin{equation}
\label{constraint}
v_{\nu}=\RD f(x_\nu,y_\nu)=  \int_\h f(x_{\nu}h y_{\nu}^{-1})\dx h=\int_{\M_{\nu}}f(h)\d h
\end{equation}
we consider the following variational problem:
for a  $t>\frac{1}{2}dim\>\G$ find a function  $u$ in the space $H_{t}(\G)$ which satisfies (\ref{constraint}) and minimizes the functional
 $$
 u\rightarrow \|(1- \Delta_\G )
^{t/2}u\|.
$$

\bigskip

According to  Theorem \ref{MainTheorem} the solution $s_{t}=s_{t}(v), \>v=\{v_{\nu}\},$ to this problem is given by the formula
\begin{equation}
\label{Splines on groups}
s_{t}=\sum_{\pi\in \widehat \G}\sum_{i,j=0}^{d_{\pi}}c_{ij}^{\pi}\pi_{ij},
\end{equation}
where 
\begin{equation}
\label{coef-c}
c_{ij}^{\pi}=c_{ij}^{\pi}(s_{t}(v))=(1+\lambda^2_\pi)^{-t}\sum_{\nu=1}^{N}\alpha_{\nu}\pi_{ij}(x_{\nu},y_{\nu}),\>\>\alpha_{\nu}=\alpha_{\nu}(s_{t}(v)).
\end{equation}
Coefficients $\alpha_{1},...,\alpha_{N}$ are solutions of the following system
\begin{equation}
\label{coef-alpha}
\beta_{1\mu}\alpha_{1}+...+\beta_{N\mu}\alpha_{N}=v_{\mu},\>\>\mu=1,...,N.
\end{equation}
To determine matrix $\beta$ with entries $\beta_{\nu\mu},\>\>\nu,\mu=1,...,N,$ one  only needs to find the data $\RD \pi(g_\nu)$ and to compute quantities $\sum_{i,j=1}^{d_\pi}\overline{\RD(\pi_{ij}(x_\nu,y_\nu))}\RD(\pi_{ij}(x_\mu,y_\mu))$. The entries $\beta_{\nu\mu}$ are given by the formulas  

\begin{align}
\label{entries-beta}
    \beta_{\mu\nu} &= \sum_{\pi\in\widehat \G}(1+\lambda^2_\pi)^{-t}\sum_{i,j=1}^{d_\pi}\overline{\RD(\pi_{ij}(x_\nu,y_\nu))}\RD(\pi_{ij}(x_\mu,y_\mu)).
\end{align}
The appearing formulae for special applications will end up in well known calculations of special function, since these arise naturally from representation theory. Implementing a fast algorithm for the solution of \eqref{eq:linsystem}, that is a standard problem, will give the solution of our variational problem for the Radon transform.

Here we have

\begin{align}
    \sum_{i,j=1}^{d_\pi}\overline{\RD(\pi_{ij}(x_\nu,y_\nu))}\RD(\pi_{ij}(x_\mu,y_\mu)) &=  \sum_{i,j=1}^{d_\pi} \int_\h \overline{\pi_{ij}(x_\nu h y_\nu^{-1})\dx h} \int_\h \pi_{ij}(x_\mu h y_\mu^{-1})\dx h\\
    &\!\!=\sum_{i,j=1}^{d_\pi} \int_\h \int_\h \pi_{ji}(y_\nu h x_\nu^{-1})\dx h \pi_{ij}(y_\mu h x_\mu^{-1})\dx h\\
    &\!\!\!\!\!= \tr(\pi_\h\pi(y_\nu)\pi_\h\pi(x^{-1}_\nu)\pi_\h\pi(x_\mu)\pi_\h\pi(y_\mu^{-1})),
\end{align}
hence we obtain function that is zonal in every component. A special case of that is the Addition theorem \ref{addition_theorem} for spherical harmonics.

In the case where $\h$ is a massive subgroup of $\G$ we find
\begin{align}
\sum_{i,j=1}^{d_\pi}\overline{\RD(\pi_{ij}(x_\nu,y_\nu))}\RD(\pi_{ij}(x_\mu,y_\mu)) &= \pi_{11}(y_\nu)\pi_{11}(x^{-1}_\nu)\pi_{11}(x_\mu)\pi_{11}(y_\mu^{-1})\\
&= \pi_{11}(y_\nu x^{-1}_\nu x_\mu y_\mu^{-1}),
\end{align}
and hence
\begin{align}
\beta_{\mu\nu} &= \sum_{\pi\in\widehat \G_1}(1+\lambda^2_\pi)^{-t} \pi_{11}(y_\nu x^{-1}_\nu x_\mu y_\mu^{-1})
\end{align}
where $\widehat \G_1$   denotes the set of irreducible representations with rank $\pi_\h=1$.

Since $SO(2)$ is a massive subgroup in $SO(3)$ the last formula can be used in the case of the Radon transform on $SO(3)$.

Let $\{(x_{1}, y_{1}),...,(x_{N}, y_{N})\}$ be a set of pairs of points from $\G$. In what follows we have to assume that our {\bf Independence Assumption} (\ref{independence}) holds. It takes now the following form: there are smooth functions $\phi_{1},...,\phi_{N}$ on $\G$  with
$$
\RD \phi_{\mu}(x_{\nu}, y_{\nu})=\delta_{\nu\mu}.
$$
But it is obvious that for this condition to satisfy it is enough to assume that submanifolds $\M_\nu=x_\nu \h y_\nu^{-1}\subset \G$ are pairwise different (not necessarily disjoint).

Let $f$ be in  a continuous function on $\G,\>\>\>$ $t> \frac{1}{2}dim \G,\>\>\>$ $v=\{v_{\nu}\}_{1}^{N}$ where
$$
v_{\nu}=\int_{\M_{\nu}}f.
$$
According to Definition \ref{interpolationdef} we use notation $s_{t}(f)=s_{t}(v)$ for a function in $ H_{t}(\G)$  such that for $\M_{\nu}=x_{\nu}\h y_{\nu}^{-1}$ one has 
\begin{equation}\label{V1}
\left(\RD  f\right)(x_{\nu}, y_{\nu})=\int_{\M_{\nu}}f=\int_{\M_{\nu}}s_{t}(f)=\left(\RD  s_{t}(f)\right)(x_{\nu}, y_{\nu}),
\end{equation}
and 
\begin{equation}\label{V2}
\|s_{t}(f)\|_{H_{t}(\G)}\rightarrow \min.
\end{equation}

\begin{thm}
\label{general group}

Let $\{(x_{1}, y_{1}),...,(x_{N}, y_{N})\}$ be a set of pairs of points from $\G$, such  that submanifolds $\M_\nu=x_\nu \h y_\nu^{-1}\subset  \G,\>\>\nu=1,...,N,$ are pairwise different.

Given a continuous function $f$ on $\G$ and a $t>\frac{1}{2} dim\>\G$ the solution of (\ref{V1})-(\ref{V2}) is given by the formula (\ref{Splines on groups}).  The Fourier coefficients $c_{k}(s_{t}(f))$ of the solution are given by their matrix entries (\ref{coef-c}), 
where $\alpha(s_{t}(f))=\left(\alpha_{\nu}(s_{t}(f))\right)_{1}^{N}\in \mathbb{R}^{N}$ is the solution of (\ref{coef-alpha}) 
with $\beta\in \mathbb{R}^{N\times N}$ given by (\ref{entries-beta}).

The function $s_{t}(f)\in H_{t}(\G)$ has the following properties:

\begin{enumerate}

\item $s_{t}(f)$ has the prescribed set of measurements
$$
\int_{\M_{\nu}}f=\int_{\M_{\nu}}s_{t}(f);
$$

\item it minimizes the functional 
 $$
 u\rightarrow \|(1- \Delta_\G )
^{t/2}u\|;
$$

\item  the solution (\ref{Splines on groups}) is optimal in the sense that for every sufficiently large $K>0$ it is the symmetry center of the convex bounded closed set of all functions $g$  in $H_{t}(\G)$  with $\|g\|_{t}\leq K$ which have the same set of measurements 
$$
\int_{\M_{\nu}}f=\int_{\M_{\nu}}g.
$$

\end{enumerate}

\end{thm}

Let us turn to the case of $SO(3)$. Since
we have (\ref{basis-action1})
\begin{align}
    \RD T^k_{ij}(\xi,\eta) &= T_{i1}^k(\xi) \overline{T^k_{j1}(\eta)} =\frac{4\pi}{2k+1} \Y_k^i(\xi)\overline{\Y_k^j(\eta)}
\end{align}
it implies
\begin{align}
    \beta_{\nu\mu} &= \sum_{k=0}^\infty (1+k(k+1))^{-t}\left(\frac{4\pi}{2k+1}\right)^2 \sum_{i,j=1}^{2k+1} \overline{\Y_k^i(x_\nu)}\Y_k^j(y_\nu) \Y_k^i(x_\mu)\overline{\Y_k^j(y_\mu)}\\
    &= \sum_{k=0}^\infty (1+k(k+1))^{-t} C_k^{\frac12}(x_\nu\cdot y_\nu) C_k^{\frac12}(x_\mu\cdot y_\mu),
\end{align}
where we made use of Addition theorem \ref{addition_theorem}. 

At  this point one uses standard methods to solve the last system, then to obtain coefficients $\alpha_{\nu}$ in (\ref{coef-alpha}), coefficients $ c_{ij}^{\pi}$ in (\ref{coef-c})
 and obtain the solution (\ref{Splines on groups}).

Consider the constrained variational problem (\ref{V1})-(\ref{V2}) for $\G=SO(3), \>\h=SO(2)$. 
Below we formulate a theorem which summarizes  our results for $SO(3)$.

\begin{thm}
Let $\{(x_{1}, y_{1}),...,(x_{N}, y_{N})\}$ be a set of pairs of points from $SO(3)$, such  that submanifolds $\M_\nu=x_\nu SO(2) y_\nu^{-1}\subset  SO(3),\>\>\nu=1,...,N,$ are pairwise different.

For a continuous function $f$ on $\G,\>\>\>$ $t>3/2,$ and a vector (of measurements) $v=\left(v_{\nu}\right)_{1}^{N}$ where
$$
v_{\nu}=\int_{\M_{\nu}}f,
$$
 the solution of (\ref{V1})-(\ref{V2}) is given by 
\begin{equation}
\label{splineSO}
s_{t}(f)=\sum_{k=0}^{\infty}\sum_{i,j=1}^{2k+1}c_{ij}^{k}(s_{t}(f))T_{ij}^{k}=\sum_{k=0}^{\infty}trace\left(c_{k}(s_{t}(f))T^{k}\right),
\end{equation}
where $T_{ij}^{k}$ are the Wigner polynomials. The Fourier coefficients $c_{k}(s_{t}(f))$ of the solution are given by their matrix entries 
\begin{equation}
c_{ij}^{k}(s_{t}(f))=\frac{4\pi}{(2k+1)(1+k(k+1))^{t}}\sum_{\nu=1}^{N}\alpha_{\nu}(s_{t}(f) \Y_k^i(x_{\nu})\overline{\Y_k^j(y_{\nu})},
\end{equation}
where $\alpha(s_{t}(f))=\left(\alpha_{\nu}(s_{t}(f))\right)_{1}^{N}\in \mathbb{R}^{N}$ is the solution of 
\begin{equation}
\beta\alpha(s_{t}(f))=f,
\end{equation} 
with $\beta\in \mathbb{R}^{N\times N}$ given by 

\begin{align}
    \beta_{\nu\mu}= \sum_{k=0}^\infty (1+k(k+1))^{-t} C_k^{\frac12}(x_\nu\cdot y_\nu) C_k^{\frac12}(x_\mu\cdot y_\mu),
\end{align}
where $C_k^{\frac12}$ are the Gegenbauer polynomials. 

The function $s_{t}(f)\in H_{t}(SO(3))$ has the following properties:

\begin{enumerate}

\item $s_{t}(f)$ has the prescribed set of measurements $v\!=\left(v_{\nu}\right)_{1}^{N}$ at points $((x_{\nu}, y_{\nu}))_{1}^{N}$;

\item it minimizes the functional 
 $$
 u\rightarrow \|(1- \Delta_{SO(3)} )
^{t/2}u\|;
$$

\item  the solution (\ref{splineSO}) is optimal in the sense that for every sufficiently large $K>0$ it is the symmetry center of the convex bounded closed set of all functions $g$  in $H_{t}(SO(3))$  with $\|g\|_{t}\leq K$ which have the same set of measurements $v=\left(v_{\nu}\right)_{1}^{N}$ at points $((x_{\nu}, y_{\nu}))_{1}^{N}$.

\end{enumerate}

\end{thm}

\section{Approximation by splines and a sampling theorem for Radon transform of bandlimited functions on $SO(3)$}

The Theorem \ref{optim} already shows that variational splines provide an optimal approximation tool.  However, much stronger approximation theorems by splines hold in the case when the set of distributions $F=\{F_{\nu}\}_{1}^{N}$ is a set
of  delta functions on  certain set of points $\{x_{\nu}\}_{1}^{N}$  of $M$ (see \cite{Pes2004}).

\begin{defn}
We will say that a finite set of points $X_{\rho}=(\xi_{1},...,\xi_{N}),\>\rho>0,$ is a
$\rho$-lattice, if

1) The balls $B(\xi_{\nu}, \rho/2)$ are disjoint.

2) The balls $B(\xi_{\nu}, \rho)$ form a cover of $M$ of a {\bf fixed }multiplicity $R_{M}$.
\end{defn}

The fact that the balls $B(\xi_{\nu}, \rho)$ form a cover of $M$ of a {\bf fixed }multiplicity $R_{M}$ means that every point of $M$ is covered by not more than $R_{M}$ balls frm this fanily.

The fact that for any manifold of bounded geometry there exist two constants  $R_{M}$ and $\rho_{0}$ such that for any $\rho<\rho_{0}$  one can construct a $\rho$-lattice was proved in \cite{Pes2000}, \cite{Pes5}, \cite{Pes2004-1}.

We will  need the following result for functions with "many" zeros from \cite{Pes2000},  \cite{Pes5}, \cite{Pes2004-1}, \cite{Pes2004}.

\begin{lem}
\label{zerolem}
There exist constants $C(M)>0, \rho(M)>0$
 such that for any $\rho<\rho(M)$, any $\rho$-lattice
  $X_{\rho}=\{\xi_{\nu}\}$ and for any $f\in H_{2d}(M)$ such that
  $f(\xi_{\nu})=0$ for all $ \xi_{\nu}\in X_{\rho},$ the following inequality holds true

$$
\|f\|\leq C(M)\rho^{2d}\|(1-\mathcal{L})^{d}f\|, d=\dim M.
$$
\end{lem}

The next theorem extends the last estimate to higher Sobolev
norms \cite{Pes2004-1}, \cite{Pes2004}.

\begin{thm}
\label{zeroThm}
There exist constants $C(M)>0,\rho(M)>0,$ such that for any
$0<\rho<\rho(M)$, any $\rho$-lattice $X_{\rho}=\{\xi_{\nu}\}$,
  any smooth $f$ which is zero on
$X_{\rho}$ and any $t\geq 0$

$$
\|(1-\mathcal{L})^{t}f\|\leq \left(C(M)\rho^{2d}\right)^{2^{m}}
\|(1-\mathcal{L})^{2^{m}d+t} f\|, t\geq 0
$$
for all $m=0,1,... .$

\end{thm}

It is important to stress, that the constant in the last inequality is growing exponentially when order of smoothness is increasing.

Now we can formulate and prove our Approximation Theorem \cite{Pes2004-1}, \cite{Pes2004}.
\begin{thm}
\label{ApproxTh}
There exist constants $C(M)>0, \>\rho(M)>0$ such that for any
$0<\rho<\rho(M)$, any $\rho$-lattice  $X_{\rho}$, any smooth
function $f$ and any $t\geq 0$ the following inequality holds true
$$
\|(1-\mathcal{L})^{t}(s_{2^{m}d+t}(f)-f)\|\leq
\left(C(M)\rho^{2}\right)^{2^{m}d} \|(1-\mathcal{L})^{2^{m}+t}f\|,
$$
for any $m=0, 1, ... .$

Moreover, if $t>d/2+k$ then there exists a $C(M,t)$ such that

$$\left\|\left(s_{2^{m}d+t}(f)(x)-f(x)\right)\right\|_{C^{k}(M)}\leq \left(C(M,t)\rho^{2}\right)
^{2^{m}d} \|(1-\mathcal{L})^{2^{m}d+t}f\|, m=0, 1, ...\>.
$$
\end{thm}

As it was mentioned,    using standard methods of interpolation theory and the fact that Besov spaces on manifolds are interpolation spaces (real interpolation method) between two Sobolev spaces the above inequalities can be extended to Besov norms.

For a compact Lie group $SO(3)$ and its closed subgroup $SO(2)$ we consider  $S^{2}\times S^{2}.$ According to Theorem \ref{RD_SO(3)-thrm} if $f\in H_{t}(SO(3))$ then $\RD f\in H_{t+1/2}^{\Delta}(S^{2}\times S^{2})$. 
Also,  integral of $f$ over the manifold $x_{\nu}SO(2) y_{\nu}^{-1}$ is the value of $\RD f$ at $(x_{\nu}, y_{\nu})$.  One can apply Theorem \ref{MainTheorem} to the manifold $S^{2}\times S^{2}$ and the set of distributions $F=\{F_{\nu,\mu}\}$ where
$$
F_{\nu, \mu}(g)=g(x_{\nu}, y_{\mu}),\>\>\>g\in C(S^{2}\times S^{2}),\>\>(x_{\nu},y_{\mu})\in S^{2}\times S^{2}
$$
to construct  corresponding interpolant $s_{\tau}(\RD f)$ of $\RD f$.  Since $dim\>\>(S^{2}\times S^{2})=4$ for the interpolant $s_{2^{m+2}+t}(\RD f)\in H_{2^{m+2}+t}(S^{2}\times S^{2})$ 
 Approximation Theorem  gives the estimate 
\begin{equation}
\|(1-2\Delta_{S^{2}\times S^{2}})^{t}(s_{ 2^{m+2}+t}(\RD f)-\RD  f)\|_{L_{2}(S^{2} \times S^{2})}\leq
$$
$$
\left(C\rho^{2}\right)^{2^{m+2}} \|(1-2\Delta_{S^{2}\times S^{2}})^{2^{m+2}+t}\RD f\|_{L_{2}(S^{2}\times S^{2})}=
$$
$$
\left(C\rho^{2}\right)^{ 2^{m+2}} \| \RD f\|_{H_{ 2^{m+2}+t}(S^{2}\times S^{2})},
\label{AppRad}
\end{equation}
for any $m=0, 1, ... .$

Let $\widehat{s}_{2^{m+2}+t}(\RD f)$ be orthogonal progection of $s_{2^{m+2}+t}(\RD f)$ onto subspace 
$$
Range  \RD=H_{1/2}^{\Delta}(S^{2}\times S^{2}).
$$
According to (\ref{basis-action1}) it means that $\widehat{s}_{2^{m+2}+t}(\RD f)$ has a representation of the form
$$
\widehat{s}_{2^{m+2}+t}(\RD f)(\xi, \eta)=\sum_{k}\sum_{ij} c_{ij}^{k}(\RD f; m, t)\Y_k^i(\xi)\overline{\Y_k^j(\eta)},\>\>\>(\xi, \eta)\in S^{2}\times S^{2}.
$$
Applying (\ref{basis-action1}) one more time we obtain that the following function defined on $SO(3)$ 
$$
S_{2^{m+2}+t}(f)(\xi)=\RD^{-1}\left(\widehat{s}_{2^{m+2}+t}(\RD f)\right)(\xi)
$$
has a representation 
$$
S_{2^{m+2}+t}(f)(\xi)=\sum_{k}\sum_{ij} \frac{2k+1}{4\pi}c_{ij}^{k}(\RD f; m, t)T^{k}_{ij}(\xi),
$$
where $T^{k}_{ij}$ are Wigner functions.

Let us stress that functions $S_{\tau}(f)$ do not interpolate $f$ in any sense. However, the following approximation results hold.

\begin{thm}\label{last-thm}
For $
t >1/2$ for sufficiently smooth functions one has the estimate
\begin{equation}
\|(1-4\Delta_{SO(3)})^{t-1/4}(S_{ 2^{m+2}+t}( f)-  f)\|_{L_{2}(SO(3))}\leq
$$
$$
\left(C\rho^{2}\right)^{2^{m+2}} \| \RD f\|_{H_{2^{m+2}+t}(S^{2}\times S^{2})}
\end{equation}
for any $m=0, 1, ... .$

In particular, if a natural $k$ satisfies the inequality $t>k+7/4$, then 

\begin{equation}
\|S_{2^{m+2}+t}( f)-  f\|_{C^{k}(SO(3))}\leq 
\left(C\rho^{2}\right)^{ 2^{m+2}} \| \RD f\|_{H_{2^{m+2}+t}(S^{2}\times S^{2})}
\end{equation}
for any $m=0, 1, ... .$

\end{thm}

The proof of the first inequality follows from (59)
and the fact that $\RD ^{-1}$ is continuous from $H_{t}^{\Delta}(S^{2}\times S^{2})$ to Sobolev $H_{t-1/2}(SO(3))$. 
The second inequality follows from  the Sobolev Embedding Theorem.

\bigskip

For an $\omega>0$ let us consider the space ${\bf E}_{\omega}(SO(3))$ of $\omega$-bandlimited functions on $SO(3)$ i.e. the span of all Wigner functions $T_{ij}^{k}$ with $k(k+1)\leq \omega$. As the formulas (\ref{eigenvalues}) and (\ref{basis-action1}) show the Radon transform of such function is $\omega$-bandlimited on $S^{2}\times S^{2}$ in the sense its Fourier expansion involves only functions $\Y_{i}^{k}\overline{\Y_{j}^{k}}$ which are eigenfunctions of $\Delta_{S^{2}\times S^{2}}$ with eigenvalue $-k(k+1)$. Under our assumption about $k$ the following Bernstein-type inequality holds for any function in the span of $\Y_{i}^{k}\overline{\Y_{j}^{k}}$
$$
\|(1-2\Delta_{S^{2}\times S^{2}})^{\tau}\RD f\|_{L^{2}(S^{2}\times S^{2})}  \leq (1+2\omega)^{\tau}\|\RD f\|_{L^{2}(S^{2}\times S^{2})} 
$$
Note that if $f\in  {\bf E}_{\omega}(SO(3))$ then the following functions are also bandlimited 
$$
\widehat{s}_{2^{m+2}+t}(\RD f)(\xi, \eta)=\sum_{k\leq \omega}\sum_{ij} c_{ij}^{k}(\RD f; m, t)\Y_k^i(\xi)\overline{\Y_k^j(\eta)},\>\>\>(\xi, \eta)\in S^{2}\times S^{2},
$$

$$
S_{2^{m+2}+t}(f)(\xi)=\sum_{k\leq \omega}\sum_{ij} \frac{2k+1}{4\pi}c_{ij}^{k}(\RD f; m, t)T^{k}_{ij}(\xi),
$$

Using our definition (\ref{SobNorm}) of the norm in the space $H_{t}(S^{2}\times S^{2})$ we obtain the following refinement of Theorem \ref{last-thm}

\begin{thm}(Sampling Theorem For Radon Transform).
\label{SamplingT}

For $f\in {\bf E}_{\omega}(SO(3))$
  and $t>1/4$  one has the estimate
\begin{equation}
\|(1-4\Delta_{SO(3)})^{t-1/4}(S_{ 2^{m+2}+t}( f)-  f)\|_{L_{2}(SO(3))}\leq
$$
$$
\left(C\rho^{2}(1+2w)\right)^{2^{m+2}}(1+2w)^{t} \| \RD f\|
\end{equation}
for any $m=0, 1, ... .$

In particular, if a natural $k$ satisfies the inequality $t>k+7/4$, then 

\begin{equation}
\|S_{2^{m+2}+t}( f)-  f\|_{C^{k}(SO(3))}\leq \left(C\rho^{2}(1+2w)\right)^{2^{m+2}}(1+2w)^{t} \| \RD f\|\end{equation}
for any $m=0, 1, ... .$

\end{thm}

Let us remind that the function $S_{2^{m+2}+t}( f)$ was constructed by using only the values of  the Radon transform $\RD f$ on a lattice of points on $S^{2}\times S^{2}$.
This Sampling Theorem shows that if 
$$
\rho<\frac{1}{\sqrt{C(1+2\omega)}}
$$
than one can have a complete reconstruction of $\omega$-bandlimited $f$ when $m$ goes to infinity by using only the values of its Radon transform $\RD f$ on any fixed $\rho$-lattice of $S^{2}\times S^{2}$.
\bigskip

{\bf Acknowledgement}

We thank Professor S.Helgason who brought our attention  to reference \cite{K}.


\begin{thebibliography}{99}
	\bibitem{ANS} 
	P. Alfeld, M. Neamtu, L.L. Schumaker, {\em Fitting scattered data on sphere-like surfaces using spherical splines},  J. Comput. Appl. Math., 73, 35--43, 1996.

\bibitem{asg} L. Asgeirsson, {\em \"Uber eine Mittelwerteigenschaft von L\"osungen homogener linearer partieller Differentialgleichungen zweiter Ordnung mit konstanten Koeffizienten.} Annals of Mathematics, \textbf{113}: 321-346, 1937.
	
	\bibitem{MMAS_I} 
	S. Bernstein, S. Ebert, {\em Wavelets on $S^3$ and $SO(3)$ -- their construction, relation to each other and Radon transform of wavelets on $SO(3)$.} Math. Methods Appl. Sci., 33(16):1895--1909, 2010. 
	
	\bibitem{bernstein/schaeben}
	S. Bernstein, H. Schaeben. {\em A one-dimensional Radon transform on $SO(3)$ and its application to texture goniometry.} Math. Methods Appl. Sci., 28:1269--1289, 2005. 
	
	\bibitem{BHS}
	S. Bernstein, R. Hielscher, H. Schaeben, {\em The generalized totally geodesic Radon transform and its application to texture analysis.} Math. Meth. Appl. Sci.; 32:379–394, 2009.
	

			\bibitem{BPS}
			K.G.van den Boogaart, R. Hielscher, J. Prestin and H. Schaeben, {\em  Kernel-based methods for inversion of the Radon
transform on SO(3) and their applications to texture analysis},  J. Comput. Appl. Math. 199 (2007),122-40

\bibitem{bunge} H.J. Bunge, {\em Texture Analysis in Material Science.} Mathematical Methods. Butterworths: London, Boston, 1982.

\bibitem{CFKT} P. Cerejeiras, M. Ferreira, U. KŠhler, and G. Teschke, {\em Inversion of the noisy Radon transform on SO(3) by Gabor frames and sparse recovery principles}, Applied and Computational Harmonic Analysis 31(3), 325-345, 2011.



	
		\bibitem{DS}
	O. Davydov, L.L. Schumaker, {\em Interpolation and scattered data fitting on manifolds using projected Powell-Sabin splines.} IMA J. Numer. Anal. 28, no. 4, 785--805, 2008.

	
	\bibitem{Ebert} S. Ebert,  {\em Wavelets and Lie groups and homogeneous spaces.} PhD thesis, TU Bergakademie Freiberg, Department of Mathematics and Computer Sciences, 2011.
	
	\bibitem{diff-wavelets}  
	S. Ebert, J. Wirth, {\em Diffusive wavelets on groups and homogeneous spaces.} Proc. Roy. Soc. Edinburgh 141A:497-520, 2011.
	
	\bibitem{F}
	G.E. Fasshauer, {\em Hermite interpolation with radial basis functions on spheres}, Adv. Comput. Math. 10, no. 1, 81--96, 1999.
	
	\bibitem{FS97}
	G.E. Fasshauer, L.L. Schumaker, {\em Scattered data fitting on the sphere. Mathematical methods for curves and surfaces I}, (Lillehammer, 1997), 117--166, Innov. Appl. Math., Vanderbilt Univ. Press, Nashville, TN, 1998.
	
	\bibitem{FS98}
	G.E. Fasshauer, L.L. Schumaker,  {\em Scattered data fitting on the sphere. Mathematical Methods for Curves and Surfaces II,} (M. Dhlen, T. Lyche,  L. L. Schumaker eds). Nashville, TN: Vanderbilt University Press, 117--166, 1998.

	  
\bibitem{J}
F. John,  {\em The ultrahyperbolic differential with four independent variables}, Duke J. Math. 4 (1938),
300-322.	
	
	\bibitem{HNW}
	
T. Hangelbroek, F. J. Narcowich,  J. D. Ward, {\em 	
    Polyharmonic and Related Kernels on Manifolds: Interpolation and Approximation}, arXiv:1012.4852.
    
    
\bibitem{HNSW}
T. Hangelbroek, F. J. Narcowich,  X. Sun,  J. D. Ward, {\em  Kernel approximation on manifolds II: the $L_{\infty}$ norm of the $L_{2}$ projector},  SIAM J. Math. Anal. 43 (2011), no. 2, 662-684.


	\bibitem{helgason} 
	S. Helgason. The Radon transform, volume 5 of Progress in Mathematics. Birkhäuser Boston Inc., Boston, MA, second edition, 1999. 
	
	\bibitem{helgason1}
	S. Helgason, Integral Geometry and Radon Transforms, Springer, New York, Dordrecht, Heidelberg, London, 2010. 
	

	\bibitem{hielscher} R. Hielscher. Die Radontransformation auf der Drehgruppe -- Inversion und Anwendung in der Texturanalyse. PhD thesis, University of Mining and Technology Freiberg, 2007. 
	
	\bibitem{HPPSS}
	R, Hielscher, D. Potts,  J. Prestin,  H. Schaeben, M. Schmalz, {\em The Radon transform on SO(3): a Fourier slice theorem and numerical inversion}, Inverse Problems 24 (2008), no. 2, 025011, 21 pp.

\bibitem{K}
T. Kakehi, and C. Tsukamoto, 1993 {\em Characterization of images of Radon transform},  Adv. Stud. Pure Math. (1993), 22, 
101-16.	
	
	\bibitem{LS}
 M.-J. Lai, L.L. Schumaker,  Spline functions on triangulations. Encyclopedia of Mathematics and its Applications, 110. Cambridge University Press, Cambridge, 2007. 
	
	
	
\bibitem{Matthies} S. Matthies, {\em On the reproducibility of the orientation distribution function of texture samples from pole figures (ghost phenomena).} Physica Status Solidi B, \textbf{92}: K135-K138, 1979.

\bibitem{Meister} L. Meister, H. Schaeben, {\em A coinceise quaternionic geometry of rotations.} Mathematical Methods in the Applied Sciences, \textbf{28}: 101-126, 2004.


\bibitem{MP}
Minakshisundaram and Pleijel, {\em Some properties of the eigenfunctions
of
the Laplace operator on Riemannian manifolds}, Can. J. Math., 1(1949),
242-256.

\bibitem{Muller} J. Muller, C. Esling, H.J. Bunge, {\em An inversion formula expressing the texture function in terms of angular distribution functions.} Journal of Physics \textbf{42}: 161-165, 1982.

\bibitem{nik} D.I. Nikolayev, H. Schaeben, {\em Characteristics of the ultra-hyperbolic differential equation govering pole density functions}, Inverse Problems, \textbf{15}: 1603-1619, 1999.


\bibitem{P}
V.P. Palamodov,  {\em Reconstruction from a sampling of circle integrals in SO(3)},  Inverse Problems 26 (2010), no. 9, 095-008, 10 pp. 


	\bibitem{Pes2000}
	 I.~Pesenson, {\em  A sampling theorem on homogeneous manifolds},
	 Trans. Amer. Math. Soc., 9(2000), 4257-4269.



	 \bibitem{PGr} 
	I. ~Pesenson, E.  ~Grinberg, {\em Invertion of the spherical Radon transform by a Poisson type formula}, Contemp. Math.,
	278, 137--147, 2001.
	
	
\bibitem{Pes5}
I.~Pesenson,  {\em Poincare-type inequalities and reconstruction
of Paley-Wiener functions on manifolds }, J. of Geometric Analysis
{\bf 4}(1), (2004), 101-121.




\bibitem{Pes2004-1}
I.~Pesenson, {\em An approach to spectral problems on Riemannian
manifolds}, Pacific J. of Math. Vol. 215(1), (2004), 183-199.


		\bibitem{Pes2004} 
		I. Pesenson. {\em Variational splines on Riemannian manifolds with applications to integral geometry.} Adv. in Appl. Math., 33(3), (2004),  548 - 572. 
		
	\bibitem{vilenkin1} N.J. Vilenkin, Special Functions and the Theory of Group Representations, Translations of Mathematical Monographs Vol. 22, American Mathematical Society, 1978.

		
	\bibitem{Vilenkin} N.J. Vilenkin and A.U. Klimyk, Representations of Lie Groups and special functions, volume~2, Kluwer Academic Publishers, 1993.
		
		\end{thebibliography}
\end{document}